\documentclass[12pt]{amsart}
\usepackage[mathscr]{eucal}
\usepackage{amssymb,amsmath}
\usepackage{amsfonts}
\usepackage{a4}

\theoremstyle{plain}
\newtheorem{thm}{Theorem}[section]
\newtheorem{prp}[thm]{Proposition}
\newtheorem{cor}[thm]{Corollary}
\newtheorem{lem}[thm]{Lemma}

\newtheorem{ex}[thm]{Example}

\theoremstyle{remark}
\newtheorem{prob}[thm]{Problem}
\newtheorem{rem}[thm]{Remark}

\numberwithin{equation}{section}

\newcommand{\RR}{\mathcal{R}}
\newcommand{\FF}{\mathcal{F}}

\newcommand{\eps}{\varepsilon}

\DeclareMathOperator{\dom}{dom}

\newcommand{\wh}{\widehat}
\newcommand{\dx}{\;{\rm d}}
\newcommand{\ol}{\overline}
\newcommand{\sub}{\subseteq}
\newcommand{\sm}{\setminus}

\newcommand{\pe}{\wp}
\newcommand{\Ce}{\protect{\mathfrak C}}
\newcommand{\Se}{\protect{\mathfrak S}}
\newcommand{\Ze}{\protect{\mathfrak Z}}

\newcommand{\E}{\protect{\mathbb E}}
\newcommand{\F}{\protect{\mathbb F}}
\newcommand{\tlo}{2^{<\omega}}
\newcommand{\vf}{\varphi}
\newcommand{\sfr}{\!\smallfrown\!}
\newcommand{\tp}{\tau_p}

\begin{document}
\vskip1cm

\title
{On Borel structures in the Banach space $C(\beta\omega)$}

\author{Witold Marciszewski}
\author[Grzegorz Plebanek]{Grzegorz Plebanek}
\address{Institute of Mathematics\\
University of Warsaw\\ Banacha 2\newline 02--097 Warszawa\\
Poland} \email{wmarcisz@mimuw.edu.pl}

\address{Instytut Matematyczny\\ Uniwersytet Wroc\l awski\\ pl.\ Grunwaldzki 2/4\\
\newline 50--384 Wroc\-\l aw, Poland}

 \email{grzes@math.uni.wroc.pl}

\date{\today}
\subjclass[2010]{Primary 46B26, 46E15, 54C35, 54H05}
\keywords{weak topology, pointwise topology, $C(K)$, Borel
structure}

\thanks{Research of the first author was partially supported by the National
Science Center research grant DEC-2012/07/B/ST1/03363}
\thanks{Research of the second author was partially supported by MNiSW
Grant N N201 418939 (2010--2013)}

\begin{abstract}
M.\ Talagrand showed that, for the \v{C}ech-Stone compactification
$\beta\omega$ of the space of natural numbers $\omega$, the norm
and the weak topology generate different Borel structures in the
Banach space $C(\beta\omega)$. We prove that the Borel structures
in $C(\beta\omega)$ generated by the weak and the pointwise
topology are also different.

We also show that in $C(\omega^*)$, where $\omega^*=\beta\omega\sm\omega$,
 there is no countable family of pointwise
Borel sets separating functions from $C(\omega^*)$.
\end{abstract}

\maketitle

\section{Introduction}\label{s1}

Given a compact space $K$, by $C(K)$ we denote the Banach space of
continuous real-valued functions $K$, equipped with the standard
supremum norm. If $K=\beta\omega$, the \v{C}ech-Stone
compactification of the space $\omega$ of natural numbers, then
$C(\beta\omega)$ is isometric to the classical Banach space
$l_\infty$.

One can consider three natural topologies on  $C(K)$:
\[\tp\sub weak\sub norm,\]
where $\tp$ is the topology of pointwise convergence.
Consequently,  one has three corresponding Borel $\sigma$-algebras
\[Borel(C(K),\tp)\sub Borel(C(K),weak) \sub Borel(C(K),norm).\]
Those three $\sigma$-algebras are equal for many classes of
nonmetrizable spaces $K$, this is the case for all spaces $K$ such
that the space $C(K)$ admits the so called pointwise Kadec
renorming, see \cite{Ed2} and \cite{Ra}, we also refer the reader
to the paper \cite[Sec.\ 3]{MP} for some comments concerning
coincidence of these $\sigma$-algebras.

On the other hand, Talagrand \cite{Ta} proved that
\[Borel(C(\beta\omega),weak)\neq Borel(C(\beta\omega), norm).\]
Marciszewski and Pol \cite{MP} showed that  $Borel(C(S),\tp)\neq
Borel(C(S), weak)$ for $S$ being the Stone space of the measure
algebra. Since, for the space $S$, the Banach spaces $C(S)$ and
$C(\beta\omega)$ are isomorphic, it follows that $C(S)$ has three
different Borel structures. Let us note that the Borel structures
in function spaces $(C(S),\tp)$ and $(C(\beta \omega ),\tp)$ are
essentially different, see Remark \ref{Stone_no_count_sep}.

We show in the present paper that $Borel(C(\beta\omega),\tp)\neq
Borel(C(\beta\omega), weak)$; our result and Talagrand's theorem
mentioned above imply that, even though $\beta\omega$ is
separable, the space $C(\beta\omega)$ possesses three different
Borel structures as well. Proving our main result, stated below as
Theorem \ref{main_thm}, we build on ideas from \cite{MP} and show
that in fact there is a measure $\wh{\mu}\in C(\beta\omega)^*$
which is not pointwise Borel measurable.

Recall that, if $\varphi:K\to L$ is a continuous surjection, then
the map $f\mapsto f\circ\varphi$ defines an embedding of $C(L)$
into $C(K)$ with respect to the norm, weak, and pointwise
topologies. Since $\beta\omega$ is a continuous image of
$\omega^*$, it follows from Theorem \ref{main_thm} that for
$\omega^*=\beta\omega\sm\omega$ one also has
\[Borel(C(\omega^*),\tp)\neq Borel(C(\omega^*), weak).\]

This result was obtained in \cite[remark 6]{MP} under some
additional set-theoretic assumption.

We show in Section \ref{s4} that no sequence of pointwise Borel
sets separates points of $C(\omega^*)$. Section \ref{s5} contains
some remarks concerning $\sigma$-fields of Baire sets in function
spaces on $\beta\omega$ and $\omega^*$.

\section{Measures on $\omega$ and $\beta\omega$}\label{s2}

We shall consider only nonnegative, finite measures. We will use
the well-known fact that any finitely additive measure $\mu$ on
$(\omega,\mathcal{P}(\omega))$ corresponds to a uniquely
determined Radon measure $\widehat{\mu}$ on $\beta\omega$ such
that $\mu(A)=\widehat{\mu}(\overline{A})$, for any
$A\in\mathcal{P}(\omega)$, where $\overline{A}$ is the closure of
$A$ in $\beta\omega$, cf. \cite{Fr}.

In the sequel, we consider only measures $\mu$ on $\omega$
vanishing on singletons; then for the corresponding measures
$\widehat{\mu}$ on $\beta\omega$, we have
$\widehat{\mu}(\omega)=0$, and we may as well treat such measures
$\widehat{\mu}$ as being defined on $\omega^*$.

The following auxiliary result can be found in  \cite[Theorem 2.2.4]{BJ}.

\begin{prp}\label{bart_judah}
If $(G_n)_n$ is a sequence of dense open subsets of $2^\omega$
then there is a sequence $(I_n)_n$ of pairwise disjoint finite
subsets of $\omega$ and a sequence of functions $\vf_n:I_n\to 2$
such that $x\in\bigcap_n G_n$ for every $x\in 2^\omega$ for which
the set $\{n\in\omega: x|I_n=\vf_n\}$ is infinite.
\end{prp}

\begin{prp}\label{BP_nonmeasurable}
No nonzero measure on $\omega$, vanishing on singletons, is
measurable with respect to the $\sigma$-algebra of subsets of
$2^\omega$ having the Baire property.
\end{prp}
\begin{proof} Suppose, towards a contradiction, that $\mu$, treated as a function on
$2^\omega$, is measurable with respect to the $\sigma$-algebra of
subsets of $2^\omega$ having the Baire property. Without loss of
generality, we can assume that $\mu(\omega)=1$. The inverse image
$\mu^{-1}(S)$ of any Borel subset $S$ of the unit interval $[0,1]$
is a tail-set with the Baire property, hence, by 0--1 Law (see
\cite{Ox}) is either meager or comeager. Observe that there exist
(necessarily unique) $t\in [0,1]$ such that $\mu^{-1}(t)$ is
comeager. Indeed, if $\mu^{-1}(1)$ is comeager, then we are done.
Otherwise, we can define inductively a sequence of integers
$k_n\le 2^n-1$, such that $\mu^{-1}([k_n/2^n,(k_n+1)/2^n))$ is
comeager for $n\in\omega$. Then the required $t$ is a unique
element of $\bigcap_{n\in\omega} [k_n/2^n,(k_n+1)/2^n)$. The map
$h\colon\mathcal{P}(\omega)\to\mathcal{P}(\omega)$, defined by
$h(A) = \omega\setminus A$, is a homeomorphism of
$\mathcal{P}(\omega)$ such that $h(\mu^{-1}(t)) = \mu^{-1}(1-t)$.
Therefore $t=1-t$, and $t=1/2$.

By Proposition \ref{bart_judah} , we have functions $\vf_n:I_n\to 2$
defined on pairwise disjoint finite sets $I_n$ such that
$\mu(A)=1/2$ whenever $\chi_A$ agrees with infinitely many $\vf_n$'s.

Let
$N_1,N_2,N_3$ be a partition of $\omega$ consisting of infinite
sets and let
\[B_i = \bigcup\{\{k: \vf_n(k)=1\}: n\in N_i\},\]
$i=1,2,3$. Then, for each $i\le 3$, $\mu(B_i)=1/2$ and the sets
$B_i$ are pairwise disjoint, a contradiction.
\end{proof}

For any subset $A$ of $\omega$ we write
\begin{eqnarray*}
\overline{d}(A)= \limsup_n\frac{|A\cap n|}n,
\end{eqnarray*}
for the outer asymptotic density of a set $A$ and
\begin{eqnarray*}
d(A)= \lim_n\frac{|A\cap n|}n,
\end{eqnarray*}
whenever the set $A$ has the asymptotic density, i.e.\ when the
above limit exists.

Given a bounded sequence $(x_n)_{n\in\omega}$ and an ultrafilter
$\pe\in\omega^*$, by $\lim_\pe x_n$ we denote the $\pe$-limit of
$(x_n)$. For any ultrafilter $\pe\in\omega^*$, we define the measure
$d_\pe$ on $\omega$ by the formula
\begin{eqnarray*}
d_\pe(A)= \lim_\pe\frac{|A\cap n|}n\,,
\end{eqnarray*}
for $A\subseteq\omega$, cf. \cite{Fr} or \cite{BFPR}.

\begin{lem}\label{d_small}
For any ultrafilter $\pe\in\omega^*$ and any $\varepsilon> 0$,
there exists a set $A\in \pe$ having asymptotic density and such
that $d(A)<\varepsilon$.
\end{lem}
\begin{proof} Take $n\ge 1$ such that $1/n<\varepsilon$ and
put $A_k= \{ni+k\colon i\in\omega\}$ for $k=0,1,\dots,n-1$. Then
there exists $k<n$ such that $A_k\in \pe$. Clearly, $d(A_k)=1/n$.
\end{proof}

We also recall the following standard fact concerning the outer
density. Here, for $A,B\subseteq\omega$, $A\subseteq^* B$ denotes,
as usual, that $A\setminus B$ is finite, and we denote by $A^*$
the set $\overline{A}\setminus A$, where $\overline{A}$ is the
closure of $A$ in $\beta\omega$.

\begin{lem}\label{small_covering}
Let $\varepsilon> 0$ and $A_n\subseteq\omega$, $n\in\omega$, be
such that $A_n\subseteq A_{n+1}$ and
$\overline{d}(A_n)<\varepsilon$ for every $n\in\omega$. Then there
exists $A\subseteq\bigcup_{n\in\omega}A_n$ such that
$\overline{d}(A)\le \varepsilon$ and $A_n\subseteq^* A$ for every
$n\in\omega$.
\end{lem}
\begin{proof} By the definition of $\overline{d}$ we have
\begin{eqnarray*} \forall n\in\omega\ \exists k_n\in\omega\
\forall k\ge k_n\ \frac{|A_n\cap k|}{k}<\varepsilon\,.
\end{eqnarray*}
Without loss of generality, we can assume that the sequence
$(k_n)$ is increasing. We define
\begin{eqnarray*} A= \bigcup_{n\in\omega} A_n\cap(k_{n+1}\setminus
k_n)\subseteq\bigcup_{n\in\omega} A_n\,.
\end{eqnarray*}
Since $A_n\subseteq A_{n+1}$, we have $(A_n\setminus k_n)\subseteq
A$, and therefore $A_n^*\subseteq A^*$ for every $n\in\omega$. For
any $k>k_0$, we have $k\in k_{n+1}\setminus k_n$, for some
$n\in\omega$, and $A\cap k \subseteq A_n\cap k$, therefore $|A\cap
k|/k<\varepsilon$, and consequently $\overline{d}(A)\le
\varepsilon$.
\end{proof}

Obviously, for the sets $A_n$ and $A$ as in the above lemma, we
have $\bigcup_{n\in\omega} A_n^*\subseteq A^*$ and ${d_\pe}(A)\le
\varepsilon$ for any ultrafilter $\pe\in\omega^*$. Let us note,
however, that this does not necessarily mean that for every
increasing sequence $A_0\sub A_1\sub\ldots\sub\omega$ such that
$d_\pe(A_n)<\eps$ there is $A$ almost containing every $A_n$ and
such that $d_\pe(A)\le\eps$. Measures on $P(\omega)$ with such an
approximation property may fail to exist, see \cite{Me} for
details.

\begin{cor}\label{null_on_sep}
For any ultrafilter $\pe\in\omega^*$, the measure $\widehat{d_\pe}$
vanishes on separable subsets of $\omega^*$.
\end{cor}
\begin{proof} Let $X$ be a subset of $\omega^*$ contained in the
closure of a set $\{\FF_n\colon n\in\omega\}\subseteq\omega^*$. Fix
$\varepsilon> 0$. For any $n\in\omega$, by Lemma \ref{d_small}, we
can pick $B_n\in \FF_n$ with
$\overline{d}(B_n)<\varepsilon/2^{n+1}$. Then, for
$A_n=\bigcup_{k\le n} B_k$, we have
$\overline{d}(A_n)<\varepsilon$, and we can apply Lemma
\ref{small_covering} for the sequence $(A_n)$, obtaining the set
$A$ satisfying $\overline{d}(A)\le \varepsilon$. For any
$n\in\omega$, we have $B_n\subseteq A_n\subseteq^* A$, hence $A\in
\FF_n$. Therefore the closure in $\omega^*$ of the set $\{\FF_n\colon
n\in\omega\}$ is contained in $A^*$, and
$\widehat{d_\pe}(X)\le\widehat{d_\pe}(A^*)\le \varepsilon$. Since
$\varepsilon$ was arbitrarily chosen, it follows that
$\widehat{d_\pe}(X)= 0$.
\end{proof}

\section{$Borel(C(\beta\omega), weak)$ and $Borel(C(\beta\omega),\tp)$ are different}\label{s3}

Let $\pe$ be a fixed ultrafilter from $\omega^*$ and let $\mu=d_\pe$ be
the measure on $P(\omega)$ defined in section \ref{s2}; we write
$\wh{\mu}$ for the corresponding Radon measure on $\beta\omega$. Then
$\wh{\mu}$ is a continuous functional on
 $C(\beta\omega)$ so in particular
$\wh{\mu}$ is measurable with respect to the $\sigma$-algebra of
weakly Borel subsets of $C(\beta\omega)$. In this section we shall
show that the measure $\wh{\mu}$ is not pointwise Borel measurable
and in this way conclude our main result. The approach presented
below builds on the technique developed by Burke and Pol
\cite{BPo} and Marciszewski and Pol \cite{MP}.

We need to fix several pieces of notation. For a set $X$, by
$[X]^{<\omega}$ we denote the family of all finite subsets of $X$,
and $X^{<\omega}$ stands for the set of all finite sequences of
elements of $X$. Given sequences $s,t\in X^{<\omega}$, $s\sfr t$
denotes their concatenation. For functions $f$ and $g$, $f\prec g$
means that the domain $\dom(f)$ of $f$ is contained in the domain
of $g$ and $g|\dom(f) = f$. We also use this notation for
sequences, treating them as functions.

Writing $2=\{0,1\}$, we denote by $C_p(\beta\omega,2)$ the space
of all continuous functions $f:\beta\omega\to 2$ equipped with the
pointwise topology.

In the sequel we consider some subsets of
$(P(\omega))^2=P(\omega)\times P(\omega)$; a typical element of
such a set is a pair $c=(A,B)$, where $A,B\sub\omega$. Given some
$c_i\in (P(\omega))^2$, we  shall use the convention for elements
that every $c_i$ can be written as $c_i=(A_i,B_i)$.

Let $\Ce$ be a subset of $(P(\omega)\times P(\omega))^\omega$ of
those sequences $c=(c_0, c_1,\ldots)$ for which the following
conditions are satisfied  for every $i$:

\refstepcounter{thm}

\begin{itemize}
\item[{\bf  \arabic{section}.\arabic{thm}(1)}] \label{conditions}
$A_i\sub A_{i+1}$, $B_i\sub B_{i+1}$, $A_i\cap B_i=\emptyset$;
\item[{\bf  \arabic{section}.\arabic{thm}(2)}] $\ol{d}(A_i),
\ol{d}(B_i)<1/6$.
\end{itemize}

We moreover denote by $\Se$ the set of all finite sequences from
$(P(\omega)\times P(\omega))^{<\omega}$ satisfying conditions
\ref{conditions}.

Given $f\in C(\beta\omega)$, $i\in \{0,1\}$, and $A\sub\omega$, we
write $f{|A}\simeq i$ if the equality $f(x)=i$ holds for
$\wh{\mu}$-almost all $x\in A^*$.

We equip $P(\omega)\times P(\omega)$ with the discrete topology
and $\Ce$ with the product topology inherited from
$(P(\omega)\times P(\omega))^\omega$. Finally, we define a
topological space $\E$ that is crucial for our considerations  as
follows
\[ \E=\{(f,c)\in C_p(\beta\omega,2)\times \Ce:
f{|A_n}\simeq 0, f{|B_n}\simeq 1\mbox{ for every } n\};\]
here $c=(c_0,c_1,c_2,\ldots)$ and $c_i=(A_i,B_i)$.

Let $\Ze$ be the set of all pairs
\[z=(z(0),z(1))\in [\omega^*]^{<\omega}\times [\omega^*]^{<\omega},\]
 such that
$z(0)\cap z(1)=\emptyset$. For $z,z'\in\Ze$ we write $z\sqsubset
z'$ to denote that  $z(0)\sub z'(0)$ and $z(1)\sub z'(1)$.

Basic open neighborhoods in $\E$ are of the form $N(\sigma,z,s)$,
where $\sigma\in\tlo$, $z\in \Ze$, $s\in\Se$, and $N(\sigma,z,s)$
is the set of all $(f,c)\in \E$ such that

\begin{itemize}
\item[(a)] $f(x)=i$ for every $x\in z(i)$, $i=0,1$; \item[(b)]
$\sigma\prec f$ and $s\prec c$.
\end{itemize}

Note that every set of the form $N(\sigma,z,s)$ is nonempty, since
$\wh{\mu}$ vanishes on singletons.

Let us say that $s\in\Se$ {\em captures} $z\in\Ze$ if, writing
$s=t\sfr (A,B)$, we have $z(0)\sub A^*$ and $z(1)\sub B^*$.

\begin{lem}\label{basic}
Every basic open set $N(\sigma,z,s)$ in $\E$ contains a
neighborhood $N(\sigma,z,s')$ where $s'$ captures $z$.
\end{lem}

\begin{proof}
Indeed, if $(A,B)$ is the final pair in $s$ then for any $\eps>0$,
using  Lemma \ref{d_small} we can find sets $C, D\sub\omega$ of
asymptotic density $<\eps$ and such that $z(0)\sub C^*$, $z(1)\sub
D^*$. Then we can put $s'=s\sfr (A',B')$, where $A'=A\cup C$,
$B'=B\cup D$ and $\eps$ is small enough.
\end{proof}

\begin{lem}\label{claim_for_lemma}
Let  $N(\sigma,z,s)$ be a basic open set  in $\E$, where
$\sigma\in 2^l$. If $G$ is a dense open subset of $N(\sigma, z,s)$
then, for every $k\ge l$, there are $m>k$, $z'\in\Ze$ with
$z\sqsubset z'$, $s'\in\Se$ with $s\prec s'$, and a function
\[\vf:I=\{i: k\le i <m\}\to 2,\]
such that for every $\tau\in 2^{k-l}$
\[N(\sigma\sfr\tau\sfr \vf,z',s')\sub G.\]
\end{lem}

\begin{proof}
Given $\tau_0\in 2^{k-l}$, we have $N(\sigma\sfr\tau_0,z,s)\cap
G\neq\emptyset$ so for some interval $I_1=\{i: k\le i<m_1\}$ and
$\vf_1: I_1\to 2$ there are $z_1 \sqsupset z$ and $s_1\succ s$
such that
\[ N(\sigma\sfr\tau_0\sfr\vf_1,z_1,s_1)\sub G.\]
Take another $\tau_1\in 2^{k-l}$. Apply the same argument for
$N(\sigma\sfr\tau_1\sfr\vf_1,z_1,s_1)$. It is clear that we arrive
at the conclusion after examining all $\tau\in 2^{k-l}$.
\end{proof}

\begin{lem}\label{main_lemma}
Let $(G_n)_{n\in\omega}$ be a decreasing sequence of open subsets
of $\E$ such that $G_0\neq\emptyset$ and every $G_n$ is dense in
$G_0$. Then there exist a sequence
$c=((A_n,B_n))_{n\in\omega}\in\Ce$,  sets $A,B\subseteq\omega$,
countable sets $Z(0),Z(1)\sub\omega^*$, and a sequence
$\vf_n:I_n\to 2$ of functions defined on pairwise disjoint finite
sets $I_n\sub\omega$ such that

\begin{itemize}
\item[(i)]$\bigcup_{n\in\omega}A_n^*\sub A^*$,
$\bigcup_{n\in\omega}B_n^*\sub B^*$, $A\cap B=\emptyset$;
\item[(ii)] ${\mu}(A), {\mu}(B)\le\frac16$;
\item[(iii)] $Z(0)\sub A^*$, $Z(1)\sub B^*$;
\item[(iv)] for every
$f\in C_p(\beta\omega,2)$ satisfying
\begin{itemize}
\item $f|A\simeq 0$, $f|B\simeq 1$,
\item $f|Z(i)=i$ for $i=0,1$,
\item $f|I_0=\vf_0$,
\item $f|I_n=\vf_n$ for infinitely many $n\ge
1$,
\end{itemize}
we have $(f,c)\in\bigcap_{n\in\omega}G_n$.
\end{itemize}
\end{lem}

\begin{proof}
Fix a basic neighborhood $N(\sigma_0,z_0,s_0)\sub G_0$; by Lemma
\ref{basic} we can assume that $s_0$ captures $z_0$. Take $k_0$
such that $\sigma\in 2^{k_0}$, set $I_0=\{0,\ldots, k_0-1\}$ and
$\vf_0=\sigma_0$.

We shall  define inductively natural numbers $k_0<k_1<k_2<\ldots$,
functions $\vf_n:I_n=\{i: k_{n-1}\le i< k_{n}\} \to 2$, pairs
$z_n\in\Ze$ with $z_0\sqsubset z_1\sqsubset\ldots$, and sequences
$s_0\prec s_1\prec\ldots$ in $\Se$ such that for every $n\ge 0$

\begin{itemize}
\item[--] $s_n$ captures $z_n$;
\item[--] for every $\tau\in 2^{k_n-k_0}$ we have
$N(\vf_0\sfr\tau\sfr\vf_{n+1},z_{n+1},s_{n+1})\sub G_{n+1}$.
\end{itemize}

Having $k_n,\vf_n,\ldots$ defined, we make the inductive step
using Lemma \ref{claim_for_lemma} for the neighborhood
$N(\vf_0,z_n,s_n)$ with $G=G_{n+1}\cap N(\vf_0,z_n,s_n)$, $l=k_0$,
and $k=k_n$ and we use $m$, $z'$, and $s'$ given by this lemma to
define $k_{n+1}$, $z_{n+1}$, and $s_{n+1}$. We complete our choice
applying Lemma \ref{basic}.

The sequence $s_n\in\Se $ defines the unique element
$c=((A_n,B_n))_{n\in\omega}\in\Ce$; we take $A,B$ applying Lemma
\ref{small_covering} to sequences $(A_n)_n$ and $(B_n)_n$ (see
also the remark following the proof of Lemma
\ref{small_covering}). We put $Z(0)=\bigcup_n z_n(0)$ and
$Z(1)=\bigcup_n z_n(1)$; note that $Z(0)\sub A^*$ and $Z(1)\sub
B^*$.

Now, if $f$ satisfies (iv) then $(f,c)\in G_n$, for infinitely
many $n$, so $(f,c)\in\bigcap_n G_n$.
\end{proof}

\begin{thm}\label{main_thm}
The measure $\wh{\mu}$ is not measurable with respect to the
pointwise Borel sets in $C(\beta\omega)$. In particular,
\[Borel(C(\beta\omega),\tp)\neq Borel(C(\beta\omega), weak).\]
\end{thm}

\begin{proof}
Suppose otherwise; then
\[F_0=\{f\in C_p(\beta\omega,2)\colon \int f\dx\wh{\mu}<1/2\},\]
is pointwise Borel in $C_p(\beta\omega,2)$. Let
$F_1=C_p(\beta\omega,2)\sm F_0$.

Let $\pi:\E\to C_p(\beta\omega,2)$ denote the projection onto the
first axis. It follows that the sets $\pi^{-1}(F_i)$ are Borel in
$\E$, so both $\pi^{-1}(F_0)$ and $\pi^{-1}(F_1)$ have the Baire
property in $\E$. Therefore, for some $i\in \{0,1\}$, there is a
decreasing sequence $(G_n)_n$ of open sets in $\E$, where
$G_0\neq\emptyset$, every $G_n$ is dense in $G_0$ and $\bigcap_n
G_n\sub \pi^{-1}(F_i)$. Take $c\in\Ce$, $A,B\sub\omega$,
$Z(0),Z(1)$, $\vf_n:I_n\to 2$ as in Lemma \ref{main_lemma}.

Let $\RR$ be an uncountable almost disjoint family of infinite subsets of $\omega$.
For $R\in\RR$ let
\[I_R=I_0\cup\bigcup_{n\in R}I_n;\]
then the family $\{I_R:R\in\RR\}$ is almost disjoint too.
Therefore there is $R\in\RR$ such that

\begin{itemize}
\item[--] $(Z(0)\cup Z(1))\cap I_R^*=\emptyset$, and
\item[--] $\mu(I_R)=0$.
\end{itemize}

Set $A'=A\sm I_R$, $B'=B\sm I_R$. Take any function $f\in
C(\beta\omega,2)$ such that $f\equiv 0$ on $A'$, $f\equiv 1$ on
$B'$ and $f$ is defined on $I_R$ so that $f|I_n=\vf_n$ for $n\in
R\cup\{0\}$. Then $f\equiv 0$ on $Z(0)$ and $f\equiv 1$ on $Z(1)$,
$f|A\simeq 0$, $f|B\simeq 1$. It follows from Lemma
\ref{main_lemma} that $(f,c)\in \bigcap_n G_n\sub \pi^{-1}(F_i)$.

On the other hand,
$f$ can be freely defined on the set
\[D=\omega\sm (A'\cup B'\cup I_R),\]
 where $\mu(D)\ge 2/3$, so
$\int f\;{\rm d}\wh{\mu}$ can take values less than $1/2$ and
greater than $1/2$, a contradiction.

\end{proof}

\section{On $C$-sets in $(C(\omega^*),\tp)$}\label{s4}

Let us recall that in a topological space $X$, the elements of the
smallest $\sigma$-algebra in $X$ containing open sets and closed
under the Souslin operation are called $C$-sets, cf.\ \cite{Ke}.
The $C$-sets are open modulo meager sets and any preimage of a
$C$-set under a continuous map is a $C$-set.

\begin{thm}\label{no_count_sep}
No countable family of $C$-sets separates the functions in the
space $(C(\omega^*),\tp)$.
\end{thm}

\begin{rem}\label{count_sep}
From the fact that $\beta\omega$ is separable, it follows easily
that $(C(\beta\omega),\tp)$ contains a countable family of open
sets separating functions.
\end{rem}

The above properties of function spaces imply immediately the
following

\begin{cor}\label{no_inject}  There is no
Borel-measurable injection $\varphi: (C(\omega^*), \tp)\to
(C(\beta \omega ), \tp)$.
\end{cor}

We keep here a part of the notation introduced in Section
\ref{s3}; in particular, we will use the space $\Ce$ and the sets
$\Se$, $\Ze$, defined in that section.

By $C_p(\omega^*,2)$ we denote the subspace of $(C(\omega^*),\tp)$
consisting of $0$\,-$1$-valued functions. The role of the space
$\E$ from the previous section will be played by the following
space
\[ \F=\{(f,c)\in C_p(\omega^*,2)\times \Ce: f|A_n^*\equiv 0,\;
f|B_n^*\equiv 1 \mbox{ for every } n\};\]
where $c=(c_0,c_1,c_2,\ldots)$ and $c_i=(A_i,B_i)$.

We will say that a pair $z\in \Ze$ and a sequence
$s=((A_0,B_0),\ldots,(A_n,B_n))\in \Se$ are consistent if
$z(0)\cap B_n^* =\emptyset = z(1)\cap A_n^*$. Clearly, if $s$
captures $z$ (see Sec.\ \ref{s3}), then $s$ and $z$ are
consistent.

Basic open neighborhoods in $\F$ are of the form $O(z,s)$, where
$z\in \Ze$ and $s\in\Se$ are consistent, and $O(z,s)$ is the set
of all $(f,c)\in \F$ such that

\begin{itemize}
\item[(a)] $f(x)=i$ for every $x\in z(i)$, $i=0,1$; \item[(b)]
$s\prec c$.
\end{itemize}

Note that the condition that $s$ and $z$ are consistent implies
that every set $O(z,s)$ is nonempty.

Repeating the proof of Lemma \ref{basic}, one easily obtains the
following

\begin{lem}\label{basic^*}
For any consistent $z\in \Ze$ and $s\in\Se$, the basic open set
$O(z,s)$ in $\F$ contains a neighborhood $O(z,s')$, where $s'$
captures $z$.
\end{lem}

The proof of Theorem \ref{no_count_sep} is based on the following
auxiliary result.

\begin{lem}\label{main_mod}
For any sequence $(X_n)_{n\in\omega}$ of $C$-sets in $\F$ there
exist a sequence $c=((A_n,B_n))_{n\in\omega}\in \Ce$ and sets
$A,B\subseteq\omega$ such that

\begin{itemize}
\item[--] $\bigcup_{n\in\omega}A_n^*\sub A^*$,
$\bigcup_{n\in\omega}B_n^*\sub B^*$, $A\cap B=\emptyset$,
\item[--] $\mu(A),\mu(B)\le\frac16$; \item[--] for any
$n\in\omega$, the set
\[\{f\in C_p(\omega^*,2)\colon f|A^*\equiv 0,\; f|B^*\equiv
1\}\times\{c\}\]
 is either contained in $X_n$ or disjoint from $X_n$.
\end{itemize}

\end{lem}

\begin{proof}
We inductively define a decreasing sequence $(V_n)_n$ of nonempty
open subsets of $\F$ and for every $n$ we choose

\begin{itemize}
\item sequences $(U^n_k)_k$  of open sets dense in $V_n$ such that
$\bigcap_{k\in\omega}U^n_k$ is either contained in $X_n$ or
disjoint from $X_n$;

\item $z_n\in \Ze$ and $s_n\in \Se$ capturing $z_n$ such that
$z_{n-1}\sqsubset z_n$, $s_{n-1}\prec s_n$, and
\[ O(z_n,s_n)\sub \bigcap_{i,k\le n}U^i_k.\]
\end{itemize}

Suppose that the construction has been carried out for $i < n$ (or
$n=0$). Since
\[ X_n\cap O(z_{n-1}, s_{n-1})\]
is a $C-$set there is a nonempty open set $V_n\sub O(z_{n-1},
s_{n-1})$ and a sequence of its dense open subsets  $(U^n_k)_k$
such that $\bigcap_{k\in\omega} U^n_k$ is either contained in
$X_n$  or disjoint from it. Then the set $G_n=\bigcap_{i,k\le n}
U^i_k$ is open and nonempty (because $V_i$ are decreasing and
$U^i_k$ are dense in $V_i$). Moreover, $G_n\sub V_n\sub O(z_{n-1},
s_{n-1})$.

Now, we can choose consistent $z_n\in\Ze$ and $s_n\in\Se$ such
that $z_{n-1}\sqsubset z_n$, $s_{n-1}\prec s_n$, and
$O(z_n,s_n)\sub G_n$; by Lemma \ref{basic^*} we can additionally
require that $s_n$ captures $z_n$. The sequence $s_0\prec
s_1\prec\ldots$ defines the unique element
$c=((A_n,B_n))_{n\in\omega}\in\Ce$. We also obtain the sets $A,B$
in the same way as in the proof of Lemma \ref{main_lemma},
applying Lemma \ref{small_covering} to sequences $(A_n)_n$ and
$(B_n)_n$.

It follows that whenever the function $f\in C_p(\omega^*, 2)$
takes values $0$ on $A^*$ and $1$ on $B^*$, the pair $(f,c)$
belongs to $\F$ and, for every $n$, $(f,c)\in O(z_n,s_n)$, since
$s_n\prec c$ and $s_n$ captures $z_n$. Therefore
\[ (f,c) \in \bigcap_{n\in\omega}\bigcap_{i,k\le n}U^i_k =\bigcap_{n\in\omega}\bigcap_{k\in\omega}
U^n_k\sub \bigcap_{k\in\omega} U^n_k,\] for every $n$, and the
lemma follows.
\end{proof}

Theorem \ref{no_count_sep} can be easily derived from the above
lemma. Indeed, if $Y_n$ are $C-$sets in $(C(\omega^*),\tp)$ then
$Z_n=Y_n\cap C_p(\omega^*,2)$ are $C-$sets in $C_p(\omega^*,2)$.
Let $\pi:\F\to C_p(\omega^*,2)$ be the projection onto the first
axis. Then $X_n=\pi^{-1}(Z_n)$ are $C-$sets in the space $\F$.
Applying Lemma \ref{main_mod} to such sets $X_n$ we conclude that
there are two different functions $g_1,g_2$ with $g_i|A^*\equiv
0$, $g_i|B^*\equiv 1$ for $i=1,2$. It follows that $(g_i,c)$ are
not separated by $X_n$ and hence $g_i$ are not separated by the
sets $Y_n$.

\begin{rem}\label{Stone_no_count_sep}
Adjusting the argument from the proof of Theorem
\ref{no_count_sep} for the setting from the paper \cite{MP}, one
can also prove that, for the Stone space $S$ of the measure
algebra, the function space $(C(S),\tp)$ has no sequence of
$C-$sets separating points (this is an unpublished result of R.\
Pol). It follows that the Banach spaces $C(S)$ and
$C(\beta\omega)$ are isomorphic, but there is no Borel-measurable
injection $\varphi: (C(S),\tp)\to (C(\beta \omega ),\tp)$.
\end{rem}

\section{Baire $\sigma$-algebras in function spaces on
$\beta\omega$ and $\omega^*$}\label{s5}
If $(X,\tau)$ is any topological space then the Baire
$\sigma$-algebra $Ba(X,\tau)$ is defined to be the smallest one
making all the continuous functions on $X$ measurable, cf.\
\cite{Ed1}, \cite{Ed2}. Recall that if $(X,\tau)$ is a locally
convex linear topological space then $Ba(X,\tau)$ is actually
generated by all continuous functionals, see \cite{Ed1}, Theorem
2.3.

For a compact space $K$,  we denote by $Ba(C(K), weak)$,
$Ba(C(K),\tp)$  the Baire $\sigma$-algebras in $C(K)$ endowed with
the weak topology, or the pointwise topology, respectively,

Theorem \ref{main_thm} implies directly the following result from
\cite[Theorem 3.4]{APR1}

\begin{thm}[Avil\'{e}s-Plebanek-Rodr\'{i}guez]\label{APR1}
$$Ba(C(\beta\omega), weak)\ne Ba(C(\beta\omega),\tp)\,.$$
\end{thm}

Using results from section \ref{s2} we can also give a simpler
proof of the above theorem:

\begin{proof} We shall show that, for any ultrafilter $\pe\in\omega^*$,
the measure $\widehat{d_\pe}$ is not
$Ba(C(\beta\omega),\tp)$-measurable. Assume the contrary.
Then there exists a countable subset $X$ of $\beta\omega$ such
that $\widehat{d_\pe}$ is measurable with respect to the
$\sigma$-algebra of subsets of $C(\beta\omega)$ generated by
$\{\delta_x\colon x\in X\}$. Corollary \ref{null_on_sep} implies
that $\widehat{d_p}$ vanishes on the closure of $X\cap\omega^*$ in
$\beta\omega$. Take $A\subseteq\omega$ such that
$X\cap\omega^*\subseteq \overline{A}$ and
$\widehat{d_\pe}(\overline{A})<1$. Let $E= \{f\in
C(\beta\omega)\colon f|\overline{A}\equiv 0\}$. Observe that
$\widehat{d_\pe}|E$ is measurable with respect to the
$\sigma$-algebra generated by $\{\delta_x\colon x\in
X\cap\omega\}$, and for any subset $C$ of $B=\omega\setminus A$
the characteristic function
$\chi_{\overline{C}}\colon\beta\omega\to \mathbb{R}$ belongs to
$E$. Then the measure $\nu\colon\mathcal{P}(\omega)\to [0,1]$
defined by
\begin{eqnarray*}
v(Z)=d_\pe(Z\cap B)=\widehat{d_\pe}(\overline{Z\cap
B})=\widehat{d_\pe}(\chi_{\overline{Z\cap B}})\,,
\end{eqnarray*}
for $Z\in\mathcal{P}(\omega)$, is nonzero, Borel-measurable and
vanishes on points of $\omega$, a contradiction with Proposition
\ref{BP_nonmeasurable}.
\end{proof}

Note finally that for any compact space $K$ we have the following inclusions
$$
\begin{array}{c c c c c}
Ba(C(K),\tp) & \subset & Borel(C(K),\tp) & \mbox{ } & \mbox{ } \\
 \cap & \mbox{ } & \cap & \mbox{ } & \mbox{ } \\
Ba(C(K),weak) & \subset & Borel(C(K), weak) & \subset & Borel(C(K),norm).
\end{array}
$$

The space $K=2^{\omega_1}$ is an example of a nonmetrizable
compactum $K$ for which all the five $\sigma$-algebras on $C(K)$
are equal, see \cite{APR2}. From previous results and the
proposition below it follows that all inclusions in the above
diagram are strict for the space $K=\beta\omega$. Since
$\beta\omega$ is a continuous image of $\omega^*$, this is also
the case for $K=\omega^*$, cf. \cite[Corollary 3.3]{APR1}.

\begin{prp}
$Bor(C(\beta\omega),\tp)\not\subseteq Ba(C(\beta\omega),weak)$.
\end{prp}

\begin{proof}
Let $\{A_\alpha:\alpha<\omega_1\}$ be a family of almost disjoint
subsets of $\omega$. For every $\alpha<\omega_1$ we pick
$\FF_\alpha\in\omega^*$ such that $A_\alpha\in \FF_\alpha$. Let us
consider the set
\[V=\{f\in C(\beta\omega): f(\FF_\alpha)>0\mbox{ for some } \alpha<\omega_1\}.\]
Then $V$ is $\tp$-open; we shall check that $V\notin Ba(C(\beta\omega),weak)$.

Suppose otherwise; then $V$ lies in the $\sigma$-algebra generated
by $\{\delta_n:n\in\omega\}$ and some family
$\{\mu_n:n\in\omega\}$, where every $\mu_n$ is a probability
measure on $\omega^*$. There is $\beta<\omega_1$ such that
$\mu_n(\ol{A_\beta})=0$ for every $n$. Let $F$ be the set of all
$0$\,-$1$-valued functions in $C(\beta\omega)$ which vanish
outside $\ol{A_\beta}$. It follows that the set $F\cap V$ lies in
the $\sigma$-algebra of subsets of $F$ which is generated by the
restrictions of $\mu_n$'s and $\delta_n$'s to $F$ which is simply
the $\sigma$-algebra generated by $\delta_n$ for $n\in A_\beta$.
On the other hand, $F\cap V=\{\chi_{\ol{N}}: N\in \FF_\beta,\;
N\sub A_\beta\}$, a contradiction, since $\FF_\beta\cap
2^{A_\beta}$ is not Borel in the Cantor set $2^{A_\beta}$.
\end{proof}

\end{document}